\documentclass[12pt,reqno,a4paper]{amsart}

\usepackage{amsmath,amssymb,amsthm,amsfonts}
\usepackage{amssymb}

\usepackage{graphicx}

\usepackage{color}

\usepackage{pgf,tikz}

\usetikzlibrary{arrows}

\numberwithin{equation}{section}

{\theoremstyle{definition}
\newtheorem{definition}{Definition}[section]}

\newtheorem{theorem}{Theorem}[section]
\newtheorem{proposition}[theorem]{Proposition}

\newtheorem{lemma}[theorem]{Lemma}

{\theoremstyle{definition}
{
\newtheorem{remark}[theorem]{Remark}

}}

\newcommand{\A}{{\mathcal{A}}}
\newcommand{\FF}{{\mathcal{F}}}
\newcommand{\SSS}{{\mathcal{S}}}
\newcommand{\ZZ}{{\mathcal{Z}}}
\newcommand{\BB}{{\mathcal{B}}}
\newcommand{\II}{{\mathcal{I}}}
\newcommand{\PP}{{\mathcal{P}}}
\newcommand{\EE}{{\mathcal{E}}}
\newcommand{\RR}{{\mathcal{R}}}

\newcommand{\Cc}{{\mathbb{C}}}
\newcommand{\Nn}{{\mathbb{N}}}
\newcommand{\Rr}{{\mathbb{R}}}
\newcommand{\Dd}{{\mathbb{D}}}

\newcommand{\la}{{\lambda}}
\newcommand{\om}{\omega}
\newcommand{\vv}{{\boldsymbol{v}}}

\def\Leb{\operatorname{Leb}}

\graphicspath{ {figs/} }

\begin{document}

\title[Dissipative Outer Billiards]
{Asymptotic periodicity in outer billiards with contraction}
\author[J.~P.~Gaiv\~ao]{Jos\'e Pedro Gaiv\~ao}


\address{Departamento de Matem\'atica e CEMAPRE, ISEG, Universidade de Lisboa, Rua do Quelhas 6, 1200-781
Lisboa, Portugal} 
\email{jpgaivao@iseg.utl.pt}



\date{\today}

\begin{abstract}
We show that for almost every $(P,\la)$ where $P$ is a convex polygon and $\la\in(0,1)$, the corresponding outer billiard about $P$ with contraction $\la$ is asymptotically periodic, i.e., has a finite number of periodic orbits and every orbit is attracted to one of them. 
\end{abstract}

\maketitle


\section{Introduction}
\label{intro}

Outer billiards about convex compact planar regions were introduced by M. Day and popularized by J. Moser \cite{D47,M78}. In this paper we are interested in the class of polygonal outer billiards which we now describe.
Let $P\subset\Cc$ be a convex $d$-gon with cyclically ordered vertices $ \boldsymbol{v} = (v_1,\dots,v_d) $. For $z\in X:=\Cc\setminus P$, let $L_z \subset\Cc$ be the supporting line of $P$ passing through $z$ such that the interior of $P$ lies on the left side of $L_z$ with respect to the obvious orientation. 
The set of points $z\in X$ for which the supporting line $L_z$ contains an edge of the polygon is denoted by $\SSS$. The set $\SSS$ is a union of $d$ half-lines. Given $z\in X\setminus\SSS$, the polygon $P$ and the supporting line $L_z$ intersect at a single vertex $v_k$. Denote by $z'$ the point in $L_z$ obtained by reflecting $z$ with respect to $v_k$. The map $T_P:X\setminus\SSS\to X$ defined by $z\mapsto z'$ is called the \textit{outer billiard map} of $P$. 
The outer billiard map is not defined on $\SSS$, and has jump discontinuities across its lines. We call $\SSS$ the \textit{singular set of $T_P$} and we shall occasional write $\SSS_P$ to express the dependency on $P$. The complement $X\setminus\SSS$ is a disjoint union of $d$ open cones $ A_{1},\ldots,A_{d} $. The apex of $A_k$ is a vertex $v_k\in P$, and the restriction $T_P|_{A_k}$ is the euclidean reflection about $v_k$. Thus, $T_P$ is a planar piecewise isometry on $d$ cones.

Polygonal outer billiards constitute a very interesting and special class of piecewise isometries~\cite{BC11,G12,GT06,T95,Sch2}. Other classes of piecewise isometries have been studied in detail, mostly in two dimensions~\cite{AsGo06,Bu01,GH95,GH97,Ha81}. Several important questions remain open regarding the dynamics of polygonal outer billiards. Notably, Neumann-Moser question about the existence of orbits accumulating on the boundary of the polygon or at infinity~\cite{M78,Gusi92,Ko89,Sch1,ViSh87}.

We will now embed the outer billiard about $P$ into a one-parameter family of dynamical systems. Let $ \lambda\in(0,1) $. Define
\begin{equation}    
\label{eq:Tlambda}
T_{P,\la}(z)=-\la z + (1+\la)v_{k} \quad \text{for } z \in A_{k}.
\end{equation}
The map $ T_{P,\la} $ is a piecewise affine contraction on the finite collection of disjoint open cones $\A_P:=\{ A_{1},\ldots,A_{d} \}$. For $\la=1$, we recover the outer billiard map $T_P$.  
We will call the map $ T_{P,\la} $ the {\em outer billiard map of $P$ with contraction $\lambda$}. See Figure~\ref{fig:poly}.

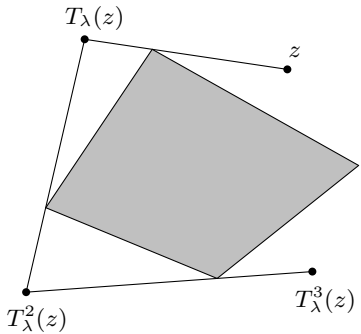
\begin{figure} 
\definecolor{qqffff}{rgb}{0,1,1}
\begin{tikzpicture}[line cap=round,line join=round,>=triangle 45,x=.7cm,y=.7cm]
\clip(-1.12,-1.54) rectangle (6.54,5);
\fill[fill=black,fill opacity=0.25] (2,4) -- (0,1) -- (3.22,-0.34) -- (5.88,1.8) -- cycle;
\draw (2,4)-- (0,1);
\draw (0,1)-- (3.22,-0.34);
\draw (3.22,-0.34)-- (5.88,1.8);
\draw (5.88,1.8)-- (2,4);
\draw (4.54,3.62)-- (0.73,4.19);
\draw (0.73,4.19)-- (-0.37,-0.6);
\draw (-0.37,-0.6)-- (5.01,-0.21);
\begin{scriptsize}
\fill  (4.54,3.62) circle (1.5pt);
\draw (4.68,3.95) node {$z$};
\fill  (0.73,4.19) circle (1.5pt);
\draw (0.92,4.6) node {$T_{\la}(z)$};
\fill  (-0.37,-0.6) circle (1.5pt);
\draw (-0.16,-1.1) node {$T^{2}_{\la}(z)$};
\fill  (5.01,-0.21) circle (1.5pt);
\draw (5.26,-.75) node {$T^{3}_{\la}(z)$};
\end{scriptsize}
\end{tikzpicture}
\caption{Example of a polygonal outer billiard with contraction. Four points of the orbit of $z$ are depicted.} 
\label{fig:poly} 
\end{figure} 

We say that a point $z \in X$ is {\em non-singular} if it has a forward orbit, i.e., $T_{P,\la}^n(z) \notin \SSS_{P} $ for every $ n \ge 0 $. For such points, we denote by $\om(z)$ the $\om$-limit set of the $ T_{P,\la} $-orbit of $ z $. The following notion is central to our study.

\begin{definition}
\label{de:ap}
We say that $T_{P,\la}$ is \textit{asymptotically periodic} if $T_{P,\la}$ has only a finite number of periodic orbits and the $\omega$-limit set of any non-singular point is a periodic orbit.
\end{definition}

Polygonal outer billiards with contraction were introduced to us by Eugene Gutkin. When the polygon is either a triangle or a parallelogram, we obtained a complete description of the dynamics using elementary methods. More precisely, for this class of polygons, we proved that the outer billiard with contraction $\la$ is asymptotically periodic and the number of periodic orbits grows as $\la\to 1$~\cite{DGG15}. We also described the sequence of bifurcations that accumulate at $\la=1$. A similar description can be obtained for the hexagon and small perturbations of parallelograms. Independently, In-Jee Jeong obtained results similar to ours using a different method~\cite{J12,J15b}. 

Regarding general results, it is known that polygonal outer billiards with contraction are asymptotically periodic when the contraction $\lambda$ is sufficiently small~\cite[Corollary 3.4]{DGG15}. On the other hand, there are uncountably many pairs $(P,\la)$ such that the outer billiard about $P$ with contraction $\la$ has an attracting Cantor set~\cite{J15}. Nevertheless, numerical experiments suggest that, for generic $(P,\la)$, the corresponding polygonal outer billiard with contraction is asymptotically periodic~\cite{J12,DGG15}. See also Figure~\ref{fig:7gon}.

\begin{figure} 
\begin{center} 
\includegraphics[width=3in]{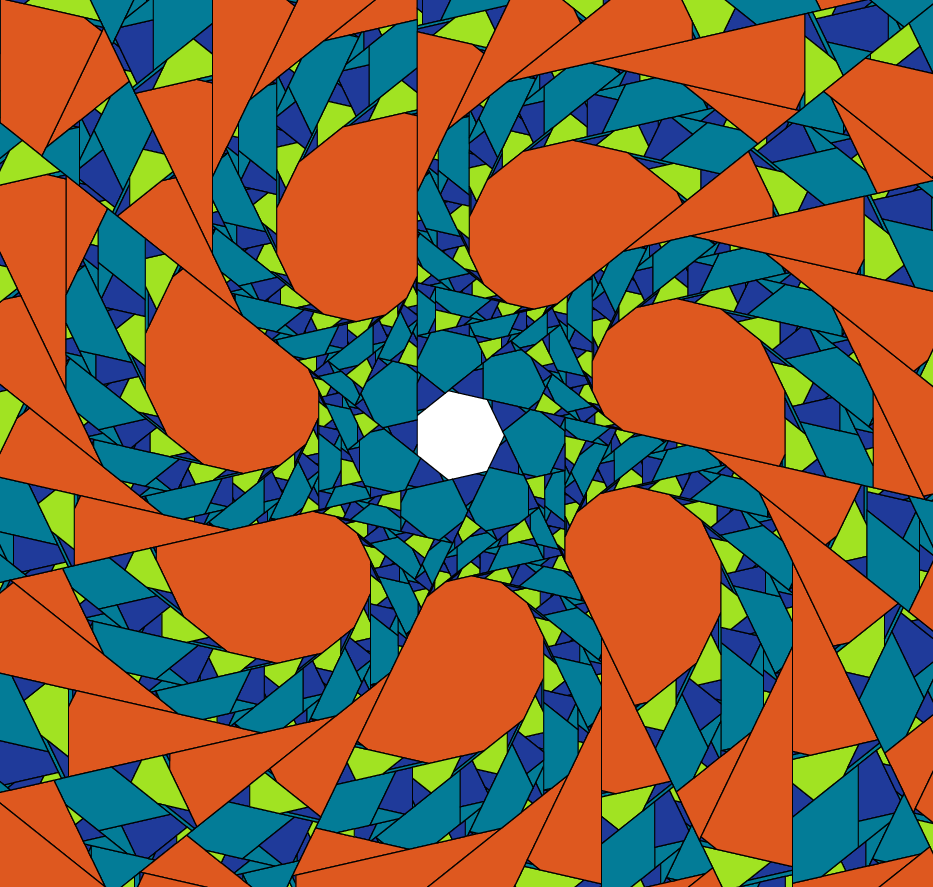} 
\end{center} 
\caption{Different basins of attraction for the outer billiard about the regular heptagon with $\lambda=0.9$.} 
\label{fig:7gon}
\end{figure}

The set $\PP_d$ of convex polygons with $d$ sides can be identified with a subset of $\mathbb{C}^d$. Naturally, it inherits the Euclidean topology and the Lebesgue measure of $\mathbb{C}^d$. The following theorem is our main result.

\begin{theorem}\label{th:main1}
For almost every pair $(P,\lambda)\in\PP_d\times(0,1)$, the outer billiard map of $P$ with contraction $\la$ is asymptotically periodic.
\end{theorem}

Polygonal outer billiards with contraction form a special class of piecewise affine contractions on the plane. The dynamics of other classes of piecewise contractions has been recently investigated in \cite{CGMU16,NP15,NPR14,BD09,B06}. In particular, it is known that any piecewise affine contracting map on the plane is asymptotically periodic for almost every choice of its branched fixed points~\cite{BD09}. Theorem~\ref{th:main1} does not follow from the results of \cite{BD09}. In polygonal outer billiards with contraction the branched fixed points are the vertices of the polygon which define the partition of the outer billiard map. Thus, moving the fixed points also moves the partition of the map, whereas in \cite{BD09} the partition is assumed to be fixed from the very beginning. 
To resolve this problem we have derived a new quantitative \L{}ojasiewicz-type inequality for polynomials of one variable (Theorem~\ref{th:polyestimate}) that may be of independent interest.

The rest of the paper is structured as follows. Sections~\ref{sec:itineraries} and \ref{sec:singular} contain some preliminary results . In section~\ref{sec:transversality} we derive a \L{}ojasiewicz-type inequality for polynomials through a concept of transversality of power series. Finally, in section~\ref{sec:mainproof} we prove Theorem~\ref{th:main1}.

\section{Itineraries}\label{sec:itineraries}
 Denote by $ \|\cdot\| $ the sup-norm in $ \Cc^{d} $. Throughout this section we fix $(P',\la')\in\PP_d\times(0,1)$.  
Any convex polygon $P\in\PP_d$ can be represented by a tuple $\vv=(v_1,\ldots,v_d)$ of vertices in the complex plane. For $\epsilon>0$ sufficiently small, any $\vv$ belonging to the ball 
$$\BB_\epsilon(P'):=\{\boldsymbol{v}\in\Cc^d\colon\|\boldsymbol{v}-\boldsymbol{v}'\|<\epsilon\}$$ 
defines a convex polygon $P\in\PP_d$ which is \textit{$\epsilon$-close to $P'$} and has the same cyclic order of its vertices. By slightly abusing the notation, we write $P\in \BB_\epsilon(P')$ and $\|P\|:=\|\vv\|$. 

Define
\begin{equation}\label{eq:ab}
a_\epsilon:=\la'+\epsilon,\qquad b_\epsilon := \sup\left\{\|P\|\colon P\in \BB_\epsilon(P')\right\},
\end{equation}
and the $\epsilon$-neighbourhood $\Delta_\epsilon$ of $(P',\la')$ by,
$$
\Delta_\epsilon=\Delta_\epsilon(P',\la'):=\BB_\epsilon(P')\times (\lambda'-\epsilon,\lambda'+\epsilon).
$$
We assume that $\epsilon>0$ is small so that $(\lambda'-\epsilon,\lambda'+\epsilon)\subset(0,1)$.

%
%
  


Denote by $ K_\epsilon := \Dd_{r_\epsilon}(0) \subset\Cc $ the closed disc in the complex plane centred at the origin and having radius $$ r_\epsilon:=b_\epsilon\frac{1+a_\epsilon}{(1-a_\epsilon)^{2}}. $$ Also let $\omega(P,\la)$ denote the union of the $\omega$-limit sets of non-singular points of $T_{P,\la}$. The following inclusions are easy to check.

\begin{lemma} \label{lem:inclusions}
\label{le1} 
For every $(P,\la)\in \Delta_\epsilon$ we have that $$\omega(P,\la)\subset T_{P,\la}(K_\epsilon) \subset K_\epsilon.$$
\end{lemma}

\begin{proof}
Given any non-singular point $ z \in X $, there exists a sequence $ (i_{j})_{j\geq0} $ with $ i_{j} \in \{1,\ldots,d\} $ such that for every $ n \in \Nn $, 
$$
T^{n}_{P,\la}(z) = (-\la)^n z + (1+\la)\sum^{n-1}_{j=0} (-\la)^{n-j-1} v_{i_{j}}.
$$
Thus,
$$
|T^{n}_{P,\la}(z)| \le a_\epsilon^n|z| + \frac{(1+a_\epsilon)}{1-a_\epsilon} b_\epsilon 
$$
and the claim follows.
\end{proof}

For every $(P,\la)\in\Delta_\epsilon$,  Lemma~\ref{lem:inclusions} tell us that the limit dynamics of $T_{P,\la}$ occurs inside the compact set $K_\epsilon$. Therefore, we restrict the open cones $A_k\in\A_P$ to $K_\epsilon$ where $T_{P,\la}$ is defined. To simplify the notation we use the same letters to denote the restricted cones of any polygon $P$ which is $\epsilon$-close to $P'$. 

Let $\A_n$ be the collection of non-empty sets
$$A_{i_0,i_1,\ldots,i_{n-1}}=A_{i_0}\cap T_{P,\la}^{-1}(A_{i_1})\cap\cdots\cap T_{P,\la}^{-n+1}(A_{i_{n-1}})$$ 
where $A_{i_k}\in\A_P$ and $i_0,i_1,\ldots,i_{n-1}\in\{1,\ldots,d\}$. Each element of $\A_n$ is a domain of continuity of $T_{P,\la}^n$. We can represent each element of $\A_n$ using the corresponding itinerary. 

\begin{definition}
The set of \textit{admissible itineraries of order $n$ of $T_{P,\la}$} is the set of $n$-tuples $(i_0,\ldots,i_{n-1})\in\{1,\ldots,d\}^n$ such that $$A_{i_0,\ldots,i_{n-1}}\neq\emptyset.$$ We denote this set by $\II_n(P,\la)$. The union of the sets $\II_n(P,\la)$ over $(P,\la)\in\Delta_{\epsilon}$ is denoted by $\II_n^\epsilon=\II_n^\epsilon(P',\la')$. 
\end{definition}

Let $T_{P,\la,k}=T_{P,\la}|_{A_k}$ denote the branch map on the $k$-th cone. Naturally, $T_{P,\la,k}$ extends to an affine contraction on all of $\Cc$.
Given $\vec{i}=(i_0,\ldots,i_{n-1})\in \II_n^\epsilon$ we define
\begin{equation*}
\begin{split}
H(P,\la,\vec{i}\,)&:=T_{P,\la,i_{n-1}}\circ\cdots\circ T_{P,\la,i_0}(0)\\
&=(1+\la)\sum^{n-1}_{j=0} (-\la)^{n-j-1} v_{i_{j}}\,,
\end{split}
\end{equation*}
and
$$
\Lambda_\epsilon(P,\la):=\bigcap_{m\geq1}\overline{\bigcup_{n\geq m}\{H(P,\la,\vec{i}\,)\colon \vec{i}\in \II_n^\epsilon\}}\,.
$$

\begin{lemma}\label{le2}
For every $(P,\la)\in \Delta_\epsilon$ we have that $$\omega(P,\la)\subset \Lambda_\epsilon(P,\la)\subset K_\epsilon.$$
\end{lemma}

\begin{proof}
Let $y\in \omega(P,\la)$. By Lemma~\ref{le1} and the definition of the omega-limit set $\omega(P,\la)$, there exists $z\in K_\epsilon$ and an increasing sequence $(n_k)_{k\geq1}$ such that $T_{P,\la}^{n_k}(z)\to y$ as $k\to\infty$. Let $(i_0,\ldots,i_{n_k-1})\in \II_{n_k}^\epsilon$ be the itinerary of $z$. Define $x_k:=H(P,\la,(i_0,\ldots,i_{n_k-1}))$. Then 
$$
|T_{P,\la}^{n_k}(z)-x_k|\leq r_\epsilon a_\epsilon^{n_k}.
$$ 
This implies that $x_k\to y$ as $k\to\infty$, i.e. $y\in\Lambda_\epsilon(P,\la)$. Finally, the inclusion $\Lambda_\epsilon(P,\la)\subset K_\epsilon$ follows from the bound
$$
|H(P,\la,\vec{i}\,)|\leq  \frac{(1+a_\epsilon)}{1-a_\epsilon} b_\epsilon \leq r_\epsilon\,.
$$
\end{proof}

\begin{lemma}\label{le3}
For every $(P,\la)\in \Delta_\epsilon$ and for every $n\geq1$, the set $\Lambda_\epsilon(P,\la)$ can be covered by a finite union of discs of radius $2r_\epsilon a_\epsilon^n$ centred at the points $\{H(P,\la,\vec{i}\,)\colon\vec{i} \in \II_n^\epsilon\}$.
\end{lemma}

\begin{proof}
Given $z\in \Lambda_\epsilon(P,\la)$, there exists an increasing sequence $(m_k)_{k\geq1}$ and itineraries $\vec{i}_{m_k}\in \II_{m_k}^\epsilon$ such that $y_{k}\to z$ as $k\to\infty$ where $y_{k}:=H(P,\la,\vec{i}_{m_k})$. Choose $k>0$ sufficiently large so that $m_k\geq n$ and $|z-y_{k}|\leq r_\epsilon a_\epsilon^n $. Define $x:=H(P,\la,(d_{0},\ldots,d_{n-1}))$ where $d_{0},\ldots,d_{n-1}$ are the last $n$ entries of $\vec{i}_{m_{k}}$. Hence, $|y_{k}-x|\leq r_\epsilon a_\epsilon^n$. This implies that $$|z-x|\leq |z-y_{k}|+|y_{k}-x|\leq 2r_\epsilon a_\epsilon^n\,.$$
\end{proof}

\begin{lemma}\label{lem:triple}
Every admissible itinerary of order $n\geq2$ has at least two distinct symbols. Moreover, there exists an integer $N\geq1$ such that for every $\epsilon>0$ sufficiently small, every itinerary of $\II_N^\epsilon$ contains at least 3 distinct symbols (labels of the vertices).
\end{lemma}

\begin{proof}
The first claim is obvious since consecutive reflections cannot take place at the same vertex of the polygon. We now prove the second claim. Suppose that for every integer $n\geq1$ there exists $\epsilon_n>0$ converging zero as $n\to\infty$ and itineraries $\vec{i}_{n}\in\II_n^{\epsilon_n}$ containing at most two symbols. In fact, by the first claim of the lemma, we can assume that $\vec{i}_n$ contains exactly two symbols. We also assume that $n$ is even. The odd case follows from similar considerations. Then
$$
\vec{i}_n=(k_n,j_n,k_n,j_n,\ldots,k_n,j_n)
$$
where $k_n,j_n\in\{1,\ldots,d\}$ represent the labels of distinct vertices. Since $\vec{i}_n$ is an admissible itinerary, it is associated with a pair $(P^{(n)},\la^{(n)})\in\Delta_{\epsilon_n}$. Notice that $(P^{(n)},\lambda^{(n)})\to(P',\la')$. Denote the vertices of $P^{(n)}$ by $\boldsymbol{v}^{(n)}=(v_{1}^{(n)},\ldots,v_{d}^{(n)})$. The affine contraction $T_{P^{(n)},\la^{(n)},j_n}\circ T_{P^{(n)},\la^{(n)},k_n}$ has a single fixed point $x^{(n)}$ which is on the line spanned by the vertices $v_{k_n}^{(n)}$ and $v_{j_n}^{(n)}$, i.e.,
$$
x^{(n)}=\frac{v_{j_n}^{(n)}-\la^{(n)} v_{k_n}^{(n)}}{1-\la^{(n)}}.
$$ 
Hence, taking $n$ sufficiently large, we conclude that the itinerary $\vec{i}_n$ cannot be admissible, which is a contradiction. This concludes the proof. 
\end{proof}


For any integer $n\geq1$, let $\mathcal{S}^n_{P,\la}$ denote the set of points $z\in X$ for which there exists an integer $0\leq k<n$ such that $T^k_{P,\la}(z)$ belongs to $\mathcal{S}_P$ (so $ \SSS^{1}_{P,\la} := \SSS_P $). We call $\mathcal{S}^n_{P,\la}$ the \textit{singular set of order $n$} of $T_{P,\la}$. Notice that $\mathcal{S}^n_{P,\la}$ is a finite union of half-lines, and $\mathcal{S}^n_{P,\la}\subseteq\mathcal{S}^{n+1}_{P,\la}$ for every $n\geq1$.

\begin{lemma}\label{lem:separation}
Let $(P,\la)\in\Delta_\epsilon$. If $$\Lambda_\epsilon(P,\la)\cap \SSS_{P}=\emptyset,$$ then $T_{P,\la}$ is asymptotically periodic.
\end{lemma}
\begin{proof}
Since $\Lambda_\epsilon(P,\la)\cap \SSS_{P}=\emptyset$, there exist $N>0$ and $\delta>0$ such that
\begin{equation}\label{eq:1}
\bigcup_{n\geq N}\{H(P,\la,\vec{i})\colon\vec{i}\in \II_n^\epsilon\}\cap \SSS_{P}^\delta=\emptyset
\end{equation}
where $\SSS_{P}^\delta$ denotes the $\delta$-neighbourhood of $\SSS_{P}\cap K_\epsilon$. 
Now let $z\in \SSS^{n+1}_{P,\la}\cap K_\epsilon\setminus\SSS^n_{P,\la}$. Denote by $\vec{i}$ the itinerary of order $n$ of $z$ and set $x:=H(P,\la,\vec{i})$. Then, $|T_{P,\la}^n(z)-x|<a_\epsilon^nr_\epsilon$. Taking $n>N$ sufficiently large we conclude that $x\in \SSS_{P}^\delta$. This contradicts \eqref{eq:1}. Thus $\SSS^{n+1}_{P,\la}\cap K_\epsilon=\SSS^n_{P,\la}\cap K_\epsilon$ for $n\geq 1$ sufficiently large. This means that $T^n_{P,\la}$,  restricted to $K_\epsilon$, maps each domain of continuity strictly inside another domain of continuity. A contraction fixed point argument finishes the proof.
\end{proof}

\section{Singular connections}\label{sec:singular}

Let $(P,\la)\in\PP_d\times(0,1)$. 
A \textit{singular connection of order $n\geq2$} corresponds to a line segment $\ell\subset \SSS_P$ which is mapped by $T^n_{P,\la}$ back to $\SSS_{P}$. The precise definition is as follows. 

\begin{definition}\label{def:singcon}
A \textit{singular connection of order $n\geq2$} is a $T_{P,\la}$-orbit segment $x_1,\ldots,x_n$ such that
\begin{enumerate}
\item $x_1,\ldots,x_{n-1}\notin\SSS_P$,
\item $x_n\in \SSS_P$,
\item the line segment $[x_1,x_n]$ contains a side of $P$.
\end{enumerate}
\end{definition}

We say that $T_{P,\la}$ has \textit{no singular connections} if there are no singular connections of any order.

\begin{lemma}\label{le5}
$T_{P,\la}$ has no singular connections for almost every $(P,\lambda)$.
\end{lemma}

\begin{proof}
The set of pairs $(P,\lambda)$ such that $T_{P,\la}$ has a singular connection of order $n\geq2$ has zero Lebesgue measure. Indeed, a singular connection $x_1,\ldots,x_n$ of order $n$ is determined by 
$$
x_n-(-\la)^n x_1 = (1+\la)\sum_{j=1}^{n-1}(-\la)^{n-1-j}v_{i_j},
$$ 
and the condition that the segment $[x_1,x_n]$ contains the $k$-th side of $P$ for some $k$. Identifying the complex plane with $\Rr^2$, this last condition can be written as 
$$
\det(x_n-x_1,v_{k+1}-v_k)=0,
$$
where $v_k$ and $v_{k+1}$ are the vertices of $P$ defining the $k$-th side. Thus,
$$
\det\left((1+\la)\sum_{j=1}^{n-1}(-\la)^{n-1-j}v_{i_j}-(1-(-\la)^n)v_k,v_{k+1}-v_k\right)=0,
$$
which defines a positive co-dimension algebraic set on the product space $\PP_d\times(0,1)$ since every admissible itinerary of order $n\geq2$ has at least two distinct symbols (Lemma~\ref{lem:triple}). So the union of all these sets has zero Lebesgue measure. 
\end{proof}



\begin{proposition}\label{prop:complexity}
If $T_{P,\la}$ has no singular connections, then there is $\epsilon>0$ such that 
$$
\lim_{n\to\infty}\frac1n\log\#\II_n^\epsilon(P,\la)=0.
$$
\end{proposition}

\begin{proof}
Let $\rho>1$ and take an integer $m>0$ sufficiently large so that $(2d)^{1/m}<\rho$. Since $T_{P,\la}$ is a homothety, the number of singular lines of $\SSS^m_P$ that meet at a common point is bounded from above by $d$. 
Let $\tau=\tau(m)>0$ be the smallest distance between any two intersection points of the lines of $\SSS^m_P$. Any point in a disc whose diameter does not exceed $\tau$ will have at most $2d$ distinct itineraries of size $m$. Because $T_{P,\la}$ has no singular connections, this property holds for every $(P',\la')\in \Delta_\epsilon(P,\la)$ for some $\epsilon>0$ sufficiently small. Moreover, choosing $\epsilon$ smaller if necessary, we can guarantee that the set of admissible itineraries of order $m$ of $T_{P',\la'}$ for any $(P',\la')\in \Delta_\epsilon(P,\la)$ coincides with the set of admissible itineraries of order $m$ of $T_{P,\la}$.

Now choose $N>0$ (depending on $m$ and $\epsilon$) large enough so that $$\text{diam}(T_{P',\la'}^N(A))<\tau,$$ for every $(P',\la')\in \Delta_\epsilon(P,\la)$ and any continuity domain $A$ of $T_{P',\la'}^N$. 

Define $a_n:=\# \II_n^\epsilon(P,\la)$. By previous observations $a_{n+m}\leq 2d\,a_n$ for every $n\geq N$. So $a_{N+im}\leq d^N (2d)^i$ for every $i\geq 0$. In other words,
$$
a_n\leq d^N (2d)^{\frac{n-N}{m}},\quad \forall\,n\geq1.
$$

Define $C:=d^N (2d)^{-N/m}$. Taking into account the choice of $m$ we get $a_n\leq C \rho^n$ for every $n\geq N$. Hence, $\log a_n\leq \log C+ n \log \rho$. Since $\rho>1$ was arbitrary, we get $\lim_n\frac1n\log a_n=0$ as we wanted to prove.

\end{proof}

\begin{remark}
The growth rate of the number of admissible itineraries is known as the \textit{singular entropy}. For a certain class of piecewise affine maps, which contains the polygonal outer billiards with contraction, the singular entropy equals the topological entropy (defined appropriately) which is known to be zero~\cite[Corollary 2]{KM06}. Because we are considering a larger set of itineraries, i.e., the set of admissible itineraries of all nearby polygonal outer billiards, Proposition~\ref{prop:complexity} does not follow directly from the results of \cite{KM06} and we had to impose a generic condition on $(P,\la)$ (see Definition~\ref{def:singcon}).
\end{remark}

\section{Transversality of power series}\label{sec:transversality}

Consider the following class of power series
$$
\FF_\alpha=\left\{p(x)=1+ \sum_{n=1}^\infty a_nx^n\colon a_n\in[-\alpha,\alpha]\right\},\quad \alpha>0
$$
and define, for each $k\geq0$
$$
r_\alpha(k):=\inf\left\{x>0\colon\exists\,p\in\FF_\alpha,\quad p^{(j)}(x)=0,\quad j=0,\ldots,k\right\}.
$$
Because $\FF_\alpha$ is compact in the topology of uniform convergence in compacta, the infimum is achieved on $\FF_\alpha$. 
Moreover, $r_\alpha(k)\geq r_\alpha(0)=1/(1+\alpha)$ where $r_\alpha(0)$ is computed using the power series $1-\alpha\sum_{n=1}^\infty x^n=1-\alpha x/(1-x)$. 

The following theorem gives a lower and upper bound for $r_\alpha(k)$.

\begin{theorem}[\cite{BBBP98}]\label{th:rk}
For every $k\geq0$,
$$
 r_\alpha(k) \leq \left(1-\frac{1}{k+2}\right)^{\min\{\alpha/9,1\}}
$$
and
$$
r_\alpha(k)\geq \left(1+\frac{1}{k+1}\right)^{-1/2}\left(\alpha^2(k+1)+1\right)^{-1/(2(k+1))}.
$$
\end{theorem}

In particular, this theorem shows that 
$$
r_\alpha(k)<1\quad\text{and}\quad\lim_{k\to\infty}r_\alpha(k)=1.
$$
The problem of estimating $r_1(1)$ was considered in \cite{S95}. A fine lower bound for $r_1(1)$ was need in order to ensure a \textit{$\delta$-transversality} property for power series in $\FF_1$ when restricted to a suitable interval. The following lemma extends this notion to $(\delta,k)$-transversality.

\begin{lemma}\label{lem:solomyak}
For every $k\geq 1$ and $\tau>0$ there exists $\delta\in(0,1)$ such that for every $p\in\FF_\alpha$ the following holds.
If $x\in[0,r_\alpha(k)-\tau]$ and $|p(x)|<\delta$, then 
$$
\max_{1\leq j \leq k}|p^{(j)}(x)|\geq \delta.
$$
\end{lemma}

\begin{proof}
We prove the claim by contradiction. Suppose that there are sequences $p_n\in\FF_\alpha$ and $x_n\in[0,r_\alpha(k)-\tau]$ such that $p_n^{(j)}(x_n)\to 0$ as $n\to\infty$ for every $j=0,\ldots,k$. Taking a subsequence, we may assume that $x_n\to x\in[0,r_\alpha(k)-\tau]$ and $p_n\to q \in\FF_\alpha$ in the topology of uniform convergence. Hence, $q^{(j)}(x)=0$ for every $j=0,\ldots,k$. But $x<r_\alpha(k)$, which contradicts the definition of $r_\alpha(k)$.
\end{proof}


\begin{remark}
This lemma, together with the fact $\lim_k r_\alpha(k)=1$, shows that $(\delta,k)$-transversality, i.e., the conclusion of the lemma, holds for points which get closer to $1$ as $k\to\infty$.
\end{remark}

The following lemma is a metric \L{}ojasiewicz-type inequality for polynomials of one variable. Its proof follows the same lines of \cite[Lemma~5.3]{K05} which in turn is based on \cite[Lemma 3]{E97}.
\begin{lemma}\label{lem:lojasiewicz}
Let $d\in\Nn$, $0<\delta<1$ and $\epsilon>0$. If $p\in\Rr[x]$ is a polynomial of degree $n\geq d$ such that for every $x\in[-1,1]$ we have,
$$
|p(x)|<\epsilon\quad\Rightarrow\quad \max_{1\leq j\leq d}|p^{(j)}(x)|\geq \delta,
$$
then
$$
\Leb\left\{x\in [-1,1] \colon |p(x)|<\epsilon\right\}\leq C \epsilon^{1/d} ,
$$
where 
$$
C:=\frac{2^{d+3}}{\delta^2}\left(4 n^{2(d+1)}\max_{x\in[-1,1]}|p(x)|+1\right).
$$
\end{lemma}
\begin{proof}
By Markov's inequality \cite[Theorem 5.1.8]{BE95} we have for $j=0,\ldots,d$ that
$$
\max_{x\in[-1,1]}|p^{(j+1)}(x)|\leq (n-j)^2 \max_{x\in[-1,1]}|p^{(j)}(x)|.
$$
Iterating this inequality gives,
\begin{align*}
A&:=\max_{0 \leq j \leq d+1}\max_{x\in[-1,1]}|p^{(j)}(x)|\\
&\leq (n-d)^2(n-d+1)^2\cdots n^2\max_{x\in[-1,1]}|p(x)|\\
&\leq n^{2(d+1)}\max_{x\in[-1,1]}|p(x)|.
\end{align*}
Now suppose that there is $x_0\in [-1,1]$ such that $|p(x_0)|< \epsilon$. Otherwise there is nothing to prove. By hypothesis $|p^{(j)}(x_0)|\geq \delta$ for some $1\leq j\leq d$. Let us suppose that $j=d$, since this is the worst case. Then, by the mean value theorem, for every $x\in I_0:=\{x\in [-1,1]\colon |x-x_0|\leq \frac{\delta}{2A}\}$ there is $\xi\in I_0$ such that
\begin{align*}
|p^{(d)}(x_0)|-|p^{(d)}(x)|&\leq |p^{(d)}(x)-p^{(d)}(x_0)|\\
&=|p^{(d+1)}(\xi)||x-x_0|\\
&\leq A|x-x_0|\leq \frac\delta2.
\end{align*}
Thus,
$$
|p^{(d)}(x)|\geq \frac\delta2,\quad \forall\,x\in I_0.
$$
Now we consider $p^{(d-1)}$ on the interval $I_0$. If there is $x_1\in I_0$ such that $|p^{(d-1)}(x_1)|<\epsilon^{1/d}$, then, again by the mean value theorem, for every $x\in I_0\setminus I_1$ where $I_1:=\{x\in I_0\colon |x-x_1|\leq \frac{4\epsilon^{1/d}}{\delta}\}$ there exists $\xi\in I_0$ such that
\begin{align*}
|p^{(d-1)}(x)|+|p^{(d-1)}(x_1)|&\geq |p^{(d-1)}(x)-p^{(d-1)}(x_1)|\\
&=|p^{(d)}(\xi)||x-x_1|\\
&\geq \frac\delta2|x-x_1|\geq 2\epsilon^{1/d}.
\end{align*}
Therefore,
$$
|p^{(d-1)}(x)|\geq \epsilon^{1/d},\quad\forall\,x\in I_0\setminus I_1.
$$
Notice that $I_1$ may be empty. Next, consider $p^{(d-2)}$ on $I_0\setminus I_1$. Clearly, $I_0\setminus I_1$ consists of two intervals which we denote by $J_2$ and $J_3$. On $J_2$, if there is $x_2\in J_2$ such that $|p^{(d-2)}(x_2)|<\epsilon^{2/d}$, then as before,
$$
|p^{(p-2)}(x)|\geq \epsilon^{2/d},
$$
for every $x\in J_2\setminus I_2$ where $I_2:=\{x\in J_2\colon |x-x_2|\leq 2\epsilon^{1/d}\}$. Similarly, there is an interval $I_3\subset J_3$ having length at most $4\epsilon^{1/d}$ such that,
$$
|p^{(d-2)}(x)|\geq \epsilon^{2/d},\quad\forall\,x\in I_0\setminus(I_1\cup I_2\cup I_3).
$$
Continuing this process, we obtain $2^d-1$ subintervals $I_i$, $i=1,\ldots,2^d-1$, of the interval $I_0$ having length at most $4\epsilon^{1/d}/\delta$ such that
$$
|p(x)|\geq \epsilon,\quad\forall\,x\in I_0\setminus(I_1\cup\cdots\cup I_{2^d-1}).
$$
Since the interval $[-1,1]$ can be partitioned in at most $4A/\delta$ intervals having the size of $I_0$ we conclude that on the whole interval we obtain at most 
$$
2^d\left(\frac{4A}{\delta}+1\right)
$$
intervals of length at most $8\epsilon^{1/d}/\delta$ where $|p(x)|<\epsilon$. 
Therefore,
\begin{align*}
\Leb\left\{x\in [-1,1] \colon |p(x)|\right\}&\leq 2^{d+3}\frac{4A+1}{\delta^2}\epsilon^{1/d}\\
&\leq \frac{2^{d+3}}{\delta^2}\left(4 n^{2(d+1)}\max_{x\in[-1,1]}|p(x)|+1\right)\epsilon^{1/d}.
\end{align*}
\end{proof}

\begin{remark}\label{rem:lojasiewicz}
We can apply Lemma~\ref{lem:lojasiewicz} to any polynomial defined in an interval $I=[a,b]\subset[0,1]$ for which the hypothesis hold on $I$. Indeed, let $g$ denote the affine transformation (orientation-preserving) that maps $[-1,1]$ to $I$ and define $\hat{p}:=p\circ g$. Then
$$
\Leb\left\{x\in I \colon |p(x)|<\epsilon\right\}=\frac{b-a}{2}\Leb\left\{x\in [-1,1] \colon |\hat{p}(x)|<\epsilon\right\}.
$$
Now we see that the polynomial $\hat{p}$ satisfies the hypothesis of the lemma. For any $x\in[-1,1]$ such that $|\hat{p}(x)|<\epsilon$ we have
\begin{align*}
\max_{1\leq j\leq d}|\hat{p}^{(j)}(x)|&= \max_{1\leq j\leq d}\left(\frac{b-a}{2}\right)^j|p^{(j)}(g(x))|\\
&\geq  \left(\frac{b-a}{2}\right)^d\max_{1\leq j\leq d}|p^{(j)}(g(x))|\\
&\geq \left(\frac{b-a}{2}\right)^d\delta.
\end{align*}
Therefore,
$$
\Leb\left\{x\in I \colon |p(x)|<\epsilon\right\}\leq C'\epsilon^{1/d}
$$
where 
$$
C':=\left(\frac{2}{b-a}\right)^{2d-1}\frac{2^{d+3}}{\delta^2}\left(4 n^{2(d+1)}\max_{x\in I}|p(x)|+1\right).
$$
\end{remark}

Denote by $\FF_\alpha^{(n)}\subset\FF_\alpha$ the set of polynomials with coefficients in $[-\alpha,\alpha]$ and having degree $n$. The following theorem is the main result of this section.

\begin{theorem}\label{th:polyestimate}
Let $I=[a,b]$ with $0\leq a<b<1$. There exist $k\geq 1$ and $0<\delta<1$ such that for every $0<\epsilon<\delta$ and every polynomial $p\in\bigcup_{n\geq0}\FF_\alpha^{(n)}$,
$$
\Leb\left\{x\in I \colon |p(x)|<\epsilon\right\}\leq C\epsilon^{1/k},
$$
where 
$$
C:=\frac{2^{3k+5}(1+\alpha)^{2k}\deg(p)^{2(k+1)}}{\delta^2(1-r_\alpha(k))}.
$$
\end{theorem}
\begin{proof}
Let $0<\tau<\min\{1-b,r_\alpha(1)-r_\alpha(0)\}$. Recall that $r_\alpha(0)=1/(1+\alpha)$. Define
$$
k:=\inf\left\{m\in\Nn\colon \frac{1}{b+\tau}>\left(1+\frac{1}{m+1}\right)^{\frac12}\left(\alpha^2(m+1)+1\right)^{\frac{1}{2(m+1)}}\right\}.
$$
By Theorem~\ref{th:rk}, $b\leq r_\alpha(k)-\tau$. Thus $I=[a,b]\subset[0,r_\alpha(k)-\tau]$. Moreover, it follows from Lemma~\ref{lem:solomyak} that there is $0<\delta<1$, depending only on $k$ and $\tau$, such for every $p\in\bigcup_{n\geq0}\FF_\alpha^{(n)}$ and every $x\in [0,r_\alpha(k)-\tau]$ we have
$$
|p(x)|<\delta\quad\Rightarrow\quad \max_{1\leq j\leq k}|p^{(j)}(x)|\geq \delta.
$$
Therefore, by  Lemma~\ref{lem:lojasiewicz} and Remark~\ref{rem:lojasiewicz},
\begin{align*}
\Leb\left\{x\in I \colon |p(x)|<\epsilon\right\}&\leq \Leb\left\{x\in [0,r_\alpha(k)-\tau] \colon |p(x)|<\epsilon\right\}\\
&\leq C\epsilon^{1/k},
\end{align*}
where 
\begin{align*}
C&:=\left(\frac{2}{r_\alpha(k)-\tau}\right)^{2k-1}\frac{2^{k+3}}{\delta^2}\left(4 \deg(p)^{2(k+1)}\max_{x\in I}|p(x)|+1\right)\\
&\leq \frac{2^{3k+2}(1+\alpha)^{2k-1}}{\delta^2}\left(\frac{4(1+\alpha) \deg(p)^{2(k+1)}}{1-r_\alpha(k)}+1\right)\\
&\leq\frac{2^{3k+5}(1+\alpha)^{2k}\deg(p)^{2(k+1)}}{\delta^2(1-r_\alpha(k))},
\end{align*}
where we have used the inequalities,
$$
\max_{x\in I}|p(x)|\leq 1+ \alpha\max_{x\in [0,r_\alpha(k)]}\sum_{i=1}^\infty x^i\leq \frac{1+\alpha}{1-r_\alpha(k)},
$$
and
$$
r_\alpha(k)-\tau>\frac{1}{1+\alpha}.
$$
\end{proof}

%

\section{Proof of Theorem~\ref{th:main1}}\label{sec:mainproof}

We want to show that the set $\ZZ$ of pairs $(P,\lambda)\in\PP_d\times(0,1)$ such that $T_{P,\lambda}$ is not asymptotically periodic has zero Lebesgue measure. 
Given any polygon $P\in\PP_d$, denote by $\EE(P)$ the set of lines extending the sides and diagonals of $P$. We say that a polygon $P\in\PP_d$ is in \textit{general position} if every pair of distinct lines in $\EE(P)$ intersect at a single point. The set of polygons in general position is denoted by $\hat{\PP}_d$. 

\begin{lemma}
Almost every polygon in $\PP_d$ is in general position.
\end{lemma}

\begin{proof}
A polygon $P\in\PP_d$ is not in general position if it has two distinct lines in $\EE(P)$ which are parallel. This imposes an algebraic condition on the set of vertices of $P$ which has positive co-dimension in $\PP_d$.
\end{proof}

By Lemma~\ref{le5}, we know that for almost every $(P,\lambda)$ the dissipative outer billiard map $T_{P,\la}$ has no singular connections. Denote by $\RR$ this full Lebesgue measure set. Therefore, it is sufficient to show that the set 
$$\ZZ':=\ZZ\cap\RR\cap (\hat{\PP}_d\times(0,1))$$ 
has zero Lebesgue measure. 

Let $(P',\lambda')\in\ZZ'$ and take $\epsilon>0$ small to be chosen during the proof. We recall the definition of the $\epsilon$-neighbourhood $\Delta_\epsilon$ of $(P',\la')$,
$$
\Delta_\epsilon=\BB_\epsilon(P')\times I_\epsilon,\quad I_\epsilon:=(\lambda'-\epsilon,\lambda'+\epsilon).
$$
Also recall the definition of $a_\epsilon$ and $b_\epsilon$ in \eqref{eq:ab}.
%
%
By Lemma~\ref{lem:separation},
\begin{align*}
\ZZ'\cap \Delta_\epsilon &= \{(P,\lambda)\in\Delta_\epsilon\cap \RR\colon  T_{P,\lambda}\text{ is not asymptotically periodic}\}\\
&\subset\{(P,\lambda)\in\Delta_\epsilon\colon \Lambda_\epsilon(P,\lambda)\cap \SSS_{P}\neq\emptyset\}.
\end{align*}
Moreover, by Lemma~\ref{le3}, for every $n\geq1$ and every $(P,\lambda)\in\Delta_\epsilon$, the limit set $\Lambda_\epsilon(P,\lambda)$ can be covered by at most $\# \II_n^\epsilon$ discs of radius $\rho_n:=2r_{\epsilon} a_{\epsilon}^n$ centred at the points $H(P,\lambda,\vec{i})$ where $\vec{i}\in \II_n^\epsilon$. Therefore,
$$
\Leb(\ZZ'\cap \Delta_\epsilon)\leq \sum_{\vec{i}\in \II_n^\epsilon}\Leb\left\{(P,\lambda)\in\Delta_\epsilon\colon \Dd_{\rho_n}(H(P,\lambda,\vec{i}))\cap \SSS_{P}\neq\emptyset\right\},
$$
where $\Dd_{\rho_n}(H(P,\lambda,\vec{i}))$ is the disc in the complex plane with radius $\rho_n$ and centred at $H(P,\lambda,\vec{i})$. Notice that
$$
S_{P}\subset\bigcup_{j=1}^dL_j(P)
$$
where $L_j(P)\in\EE(P)$ denotes the supporting line of $P$ containing the $j$-th side, i.e., joining the vertices $v_j$ and $v_{j+1}$ of $P$. Thus,
$$
\Leb(\ZZ'\cap \Delta_\epsilon)\leq \sum_{\vec{i}\in \II_n^\epsilon}\sum_{j=1}^d \Leb \Omega_\epsilon(n,\vec{i},j),
$$
where 
$$
\Omega_\epsilon(n,\vec{i},j):= \left\{(P,\lambda)\in\Delta_\epsilon\colon H(P,\lambda,\vec{i})\in L_j^{\rho_n}(P)\right\},
$$
and $L_j^{\rho_n}(P)$ is the $\rho_n$-neighbourhood of the line $L_j(P)$. 



\begin{lemma} There exists $\epsilon>0$ and $k\geq1$, $N>0$ and $C>0$ depending only on $\epsilon$ such that for every $n> N$, $\vec{i}\in\II_n^\epsilon$ and $j\in\{1,\ldots,d\}$,
$$
\Leb\Omega_\epsilon(n,\vec{i},j)\leq C n^{2(k+1)}\rho_n^{1/k}.
$$
\end{lemma}
\begin{proof}
Let $(P,\lambda)\in\Delta_\epsilon$ and $\vec{i}\in\II_n^\epsilon$. Identifying the complex plane with $\Rr^2$, the condition $H(P,\lambda,\vec{i})\in L_j^{\rho_n}(P)$ can be written in the following way,
$$
\left|\langle H(P,\lambda,\vec{i})-v_j,\eta_j\rangle\right|<\rho_n,
$$
where $\eta_j$ is a unit vector perpendicular to $L_j(P)$. 
Using the definition of $H(P,\la,\vec{i})$, a simple computation shows that,
\begin{align*}
h_j(P,\la,\vec{i}) &:=\langle H(P,\lambda,\vec{i})-v_j,\eta_j\rangle\\
&  = c_0- c_1 \la + c_2 \la^2 + \cdots + c_n (-\la)^n,
\end{align*}
where
\begin{align*}
c_0&=\langle v_{i_{n-1}}-v_j,\eta_j \rangle,\\
c_\ell&=\langle v_{i_{n-\ell-1}}-v_{i_{n-\ell}},\eta_j \rangle,\quad \ell=1,\ldots,n-1,\\
c_n&= -\langle v_{i_{0}},\eta_j \rangle.
\end{align*}
Notice that,  $|c_\ell|\leq 2\|P\|\leq 2b_\epsilon$ for every $0\leq \ell \leq n$.
Hence, $h_j(P,\cdot,\vec{i})$ is a polynomial of degree at most $n$ with coefficients belonging to the interval $[-2 b_\epsilon,2 b_\epsilon]$. 

Since $P'$ is in general position, any polygon $P\in\BB_\epsilon(P')$ is also in general position for $\epsilon>0$ sufficiently small. This implies that there is a positive constant $\tau_\epsilon=\tau_\epsilon(P')>0$, which can be explicitly computed from the angles of the sides and diagonals of $P'$, such that
$$
c_\ell\neq0\quad\Rightarrow\quad |c_\ell|\geq \tau_\epsilon,\quad 0\leq \ell\leq n-1.
$$

Taking $\epsilon>0$ smaller if necessary, Lemma~\ref{lem:triple} gives the existence of an integer $N>0$ such that every itinerary of $\II_N^\epsilon$ contains at least 3 distinct vertices. This implies that for every $n>N$ some coefficient $c_\ell$ of $h_j(P,\la,\vec{i})$ is non-zero for $0\leq \ell \leq N$. Therefore,
$$
h_j(P,\la,\vec{i})=c_\ell(-\la)^\ell \hat{h}_j(P,\la,\vec{i}),
$$
where $\hat{h}_j(P,0,\vec{i})=1$, $|c_\ell|\geq\tau_\epsilon$ and the coefficients of the polynomial $\hat{h}_j(P,\cdot,\vec{i})$ belong to the interval $[- \alpha_\epsilon, \alpha_\epsilon]$ with $\alpha_\epsilon:=2 b_\epsilon/\tau_\epsilon$. This means that $\hat{h}_j(P,\cdot,\vec{i})\in\FF_{\alpha_\epsilon}$.

Now, for every $n> N$, if $ |h_j(P,\la,\vec{i})|<\rho_n$ then
$$
|\hat{h}_j(P,\la,\vec{i})|=\frac{|h_j(P,\la,\vec{i})|}{|c_\ell | \la^\ell}<\frac{\rho_n}{\tau_\epsilon \epsilon^N}.
$$
Applying Theorem~\ref{th:polyestimate}, there are $k\geq1$ and $0<\delta<1$ depending only on $I_\epsilon$ such that taking $N>0$ sufficiently large so that $\rho_N<\delta$ we get
$$
\Leb \left\{\la\in I_\epsilon\colon |h_j(P,\la,\vec{i})|<\rho_n\right\}\leq C\, n^{2(k+1)} \rho_n^{1/k},\quad \forall\,n> N
$$
where 
$$
C:=\frac{2^{3k+5}(1+\alpha_\epsilon)^{2k}}{(\tau_\epsilon \epsilon^N)^{1/k}\rho_N^2(1-r_{\alpha_\epsilon}(k))}.
$$

Finally, in order to complete the proof we can use Fubini's theorem and obtain the desired estimate for $\Leb\Omega_\epsilon(n,\vec{i},j)$. 
\end{proof}

Applying this lemma, there is $\epsilon>0$, $N>0$ and $k\geq1$ such that,
$$
\Leb(\ZZ'\cap\Delta_\epsilon)\leq d C n^{2(k+1)}\rho_n^{1/k}\# \II_n^\epsilon,\quad \forall n>N
$$
where the constant $C>0$ is independent of $n$. By Lemma~\ref{prop:complexity}, $\lim_n1/n\log\# \II_n^\epsilon=0$. Therefore, we conclude that
$$
\Leb(\ZZ'\cap\Delta_\epsilon)\leq \lim_{n\to\infty}d C n^{2(k+1)}\rho_n^{1/k}\# \II_n^\epsilon = 0,
$$
because $\rho_n$ decreases exponentially at rate $a_\epsilon$. Since $(P',\lambda')\in\ZZ'$ was arbitrary, this concludes the proof of Theorem~\ref{th:main1}.
\qed

\section*{Acknowledgements}


The author was supported through the FCT/MEC grant SFRH/BPD/78230/2011 and partially supported by the Project CEMAPRE - UID/MULTI/00491/2013 financed by FCT/MCTES through national funds. The author also wishes to express his gratitude to Gianluigi Del Magno for stimulating conversations and to Pedro Duarte for the idea of using \L{}ojasiewicz-type inequalities.

\bibliography{outerlib}
\bibliographystyle{plain}

\end{document}